\documentclass[a4paper,11pt]{amsart}
\usepackage{amsmath,amsthm,amssymb}
\usepackage{graphicx}
\usepackage{mathrsfs}
\usepackage{hyperref}
\usepackage{mathtools}
\usepackage{microtype}
\usepackage{enumerate}
\usepackage{xcolor}
\usepackage{thmtools,thm-restate}
\usepackage{cleveref}
\usepackage[noadjust]{cite}
\usepackage{comment}
\hypersetup{colorlinks=true,linkcolor=blue,filecolor=blue,citecolor=blue}
\newtheorem{theorem}{Theorem}[section]
\newtheorem*{theorem*}{Theorem}
\newtheorem{lemma}[theorem]{Lemma}
\newtheorem*{lemma*}{Lemma}
\newtheorem{corollary}[theorem]{Corollary}
\newtheorem{conjecture}[theorem]{Conjecture}

\newtheorem{proposition}[theorem]{Proposition}
\newtheorem*{proposition*}{Proposition}

\newtheorem{claim}{Claim}

\newtheorem{fact}{Fact}
\newtheorem{remark}{Remark}
\DeclareMathOperator{\WReach}{WReach}
\DeclareMathOperator{\wcol}{wcol}
\DeclareMathOperator{\girth}{girth}
\DeclareMathOperator{\mad}{mad}
\DeclareMathOperator{\avd}{\bar{\rm d}}
\DeclareMathOperator{\arb}{arb}
\DeclareMathOperator{\ap}{a^{\mathclap{\prime}}}

\DeclareMathOperator{\A}{A}
\newcommand{\W}[1][-1]{\ensuremath \mathop{W_{#1}}}

\newenvironment{clproof}{ \trivlist
        \item[\hskip\labelsep
        \emph{Proof of the claim}.]\ignorespaces
}{\hfill$\vartriangleleft$\medskip
        
}
\usepackage{cleveref}
\crefname{conjecture}{Conjecture}{Conjectures}
\crefname{corollary}{Corollary}{Corollaries}
\crefname{lemma}{Lemma}{Lemmas}
\crefname{theorem}{Theorem}{Theorems}
\crefname{problem}{Problem}{Problems}
\crefname{figure}{Figure}{Figures}
\newcommand{\ERCagreement}{
        {\footnotesize
        The first author is supported by the China Scholarship Council (CSC)	and  SEU Innovation Capability Enhancement Plan for Doctoral Students (CXJH\_SEU 24119).}\\[4pt]
		\noindent\begin{minipage}{.73\textwidth}
			\footnotesize
    This paper is part of a project that has received funding from the European Research Council (ERC) under the European Union's Horizon 2020 research and innovation program (grant agreement No 810115 -- {\sc Dynasnet}).
		\end{minipage}\hfill
		\begin{minipage}{.25\textwidth}
\phantom{.}\hfill\includegraphics[height=13mm]{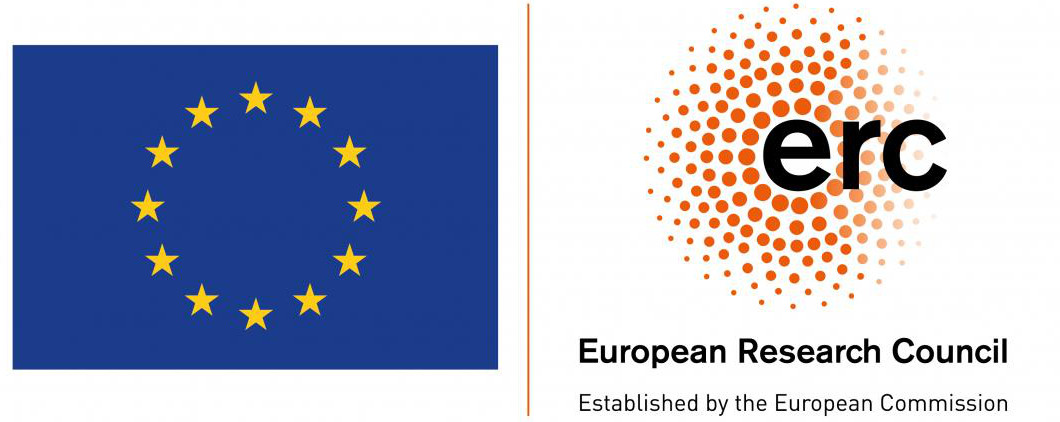}\hfill\phantom{.}
		\end{minipage}
}

\begin{document}

\title{Path degeneracy and applications}
\thanks{\ERCagreement}
\author[Y. Lin]{Yuquan Lin}
\address{Southeast University, Nanjing, Jiangsu, China \and Centre d'Analyse et de Mathématique Sociales CNRS UMR 8557, France.}
\email{yqlin@seu.edu.cn}
\author[P. Ossona de Mendez]{Patrice Ossona de Mendez}
\address{Centre d'Analyse et de Mathématique Sociales CNRS UMR 8557, France \and Computer Science Institute of Charles University (IUUK), Praha, Czech Republic}
\email{pom@ehess.fr}
\begin{abstract}
In this work, we relate girth and path-degeneracy in classes with sub-exponential expansion, with explicit bounds for classes with polynomial expansion and proper minor-closed classes that are tight up to a constant factor (and tight up to second order terms if a classical conjecture on existence of $g$-cages is verified). As an application, we derive bounds on the generalized acyclic indices, on the generalized arboricities, and on the weak coloring numbers of high-girth graphs in such classes.
Along the way, we prove a conjecture 
proposed in 
[T.~Bartnicki \emph{et al.}, \emph{Generalized arboricity of graphs with large girth},
  Discrete Mathematics \textbf{342} (2019), no.~5, 1343--1350.],~which asserts that, for every integer $k$, there is an integer $g(p,k)$ such that every $K_k$ minor-free graph with girth at least $g(p,k)$  has $p$-arboricity at most $p+1$. 
\end{abstract}
\maketitle

 \section{Introduction}

The original notion of path-degeneracy
was introduced in \cite{NO2015} in a study of the circular chromatic number of graphs with large girth excluding a minor. 
In this paper, we consider a slightly different definition, which will be suitable in a broader context (bounding the generalized $p$-arboricity, the generalized $p$-acyclic chromatic number, or the weak coloring number of graphs with large girth in a ``nice'' monotone class $\mathscr C$).

A graph is  \emph{$p$-path degenerate} if it can be constructed from the empty graph by successively adding an isolated vertex, a vertex of degree one, or a path of length $p$ with both endpoints in the previously constructed graph (see \Cref{sec:def,fig:PD}).
\begin{figure}[h!]
\begin{center}
\includegraphics[width=.75\textwidth]{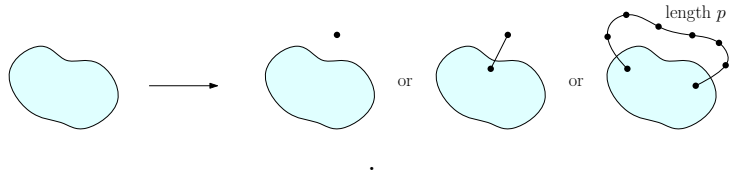}
\end{center}.
    \caption{Construction of a $p$-path degenerate  graph}
    \label{fig:PD}
\end{figure}

In this paper, we denote by $|G|$ the \emph{order} of a graph $G$ (that is, its number of vertices), and by $\|G\|$ its \emph{size}
(that is, its number of edges).
Recall that a  minor  of a graph $G$  at \emph{depth} $r$ is a minor of $G$ obtained by contracting disjoint subgraphs with radius at most $r$ and removing edges, and that 
$\nabla_r(G)$ denotes the maximum ratio $\|H\|/|H|$ over all (non-empty) minors $H$ of $G$ at depth $r$.
The \emph{expansion function}
${\rm Exp}_{\mathscr C}$ of a class $\mathscr C$ is defined by
\[{\rm Exp}_{\mathscr C}(r)
=\sup\{\nabla_r(G)\colon G\in\mathscr C\}.\]

A class $\mathscr C$ has 
\begin{itemize}
    \item \emph{bounded expansion},  if ${\rm Exp}_{\mathscr C}(r)<\infty$;
    \item  \emph{sub-exponential expansion}, if ${\rm Exp}_{\mathscr C}(r)=2^{o(r)}$;
    \item \emph{polynomial expansion},  if ${\rm Exp}_{\mathscr C}(r)=r^{O(1)}$.
\end{itemize}
Formal definitions and basic facts about shallow minors will be recalled in~\Cref{Sec:Pre}, and we refer the interested reader to \cite{Sparsity} for an in-depth study.

\subsection{Our results}

As a warm-up, we consider the general case of classes with sub-exponential expansion and prove the following.
\begin{restatable}{theorem}{ThmSE}
 \label{th:pd-subexponential}
    Let $\mathscr C$ be a class with sub-exponential expansion.
    Then, for every integer $p$ there exists an integer $g_p$ such that every graph $G\in\mathscr C$ with girth at least $g_p$ is $p$-path degenerate.
\end{restatable}

Then, we consider more precisely the case of classes with polynomial expansion, and provide a bound on $g_p$  in terms of the $\W$ branch of Lambert's $W$ function.

\begin{restatable}{theorem}{ThmPol}
\label{th:pd-polynomial}
      Let $\mathscr C$ be a class with polynomial expansion, with ${\rm Exp}_{\mathscr C}(r)\leq a(r+\frac12)^b$ for some positive reals
 	$a,b$.
We define
\[
g_p=\max\left(7,2\left\lfloor-\frac{2b}{\log 2}\W\left(-\frac{\log 2}{(24\sqrt{2}a)^{1/b}\,bp}\right)\right\rfloor+4\right)(p-1).
\]

    Then, every graph $G\in\mathscr C$ with
\[
\girth(G)>g_p\sim 4bp\log_2p\\
\]
is $p$-path degenerate.
\end{restatable}

Then, we prove that our bound is tight up to a constant factor.

\begin{restatable}{theorem}{ThmPolLower}
\label{thm:poly_lower}
    For every real $b>0$, there exists a class with expansion $O(r^b)$ such that  for infinitely many $p\in\mathbb N$ the class  contains a graph $G$ with
\[
\girth(G) \ge -\frac{8b}{3\log 2}\W\left(-\frac{\log 2}{(p-1)b}\right)(p-1)\sim\frac{8}{3}bp\log_2 p
\]
that is not $p$-path degenerate.
\end{restatable}

We conjecture that the bound given in \Cref{th:pd-polynomial} is tight up to second order term.

\begin{restatable}{conjecture}{ConjPolLower}
\label{conj:poly_lower}
    For every real $b>0$, there exists a class with expansion $O(r^b)$ such that  for infinitely many $p\in\mathbb N$ the class  contains a graph $G$ with
\[
\girth(G) \sim 4bp\log_2 p
\]
that is not $p$-path degenerate.
\end{restatable}

As a support to this conjecture, we note that its validity would follow from the following conjecture on the existence of $g$-cages (which will be used to support our other conjectures, too).

\begin{conjecture}[{\cite[p. 164]{B1978}}]
\label{conj:cage}
For infinitely many integers $g$ there
exists a cubic graph with girth at least $g$ and order at most $c2^{g/2}$.
\end{conjecture}

Then, we consider proper minor-closed classes of graphs and prove the following bound.

\begin{restatable}{theorem}{ThmMC}
  \label{th:pd-MC}
    Let $\mathscr{C}$ be a  proper minor-closed class of graphs and let $d$ be the maximum average degree of the graphs in $\mathscr C$.
    Then, for every integer $p\ge2$, every graph $G\in\mathscr{C}$ with %
    $$\girth(G)>(4\log_{2}d+2\log_{2}(\min\{d,576\})+3)(p-1)$$
    is $p$-path degenerate.  
\end{restatable}

Again, we prove that our bound is tight up to a constant factor.

\begin{restatable}{theorem}{ThmMClower}
\label{th:MC_lower}
There exists a constant $c$ with the following property:
for every real $d_0>0$ there exists a real $d\geq d_0$ and a proper minor closed class  with maximum average degree at most $d$ that contains,  for every integer $p\geq 2$, a graph $G_p$ with 
\[
\girth(G_p)\ge\left(\frac83\log_2(d-2)+c\right)(p-1)
\]
that is not $p$-path degenerate.
\end{restatable}

We conjecture that the bound given in \Cref{th:pd-MC} is essentially tight, what would be a consequence of  the validity of \Cref{conj:cage}.

\begin{restatable}{conjecture}{ConjMClower}
\label{conj:MC_lower}
There exists a constant $c$ with the following property:
for every real $d_0>0$ there exists a real $d\geq d_0$ and a proper minor closed class  with maximum average degree at most $d$ that contains,  for every integer $p\geq 2$, a graph $G_p$ with 
\[
\girth(G_p)\ge(4\log_2(d-2)+c)(p-1)
\]
that is not $p$-path degenerate.
\end{restatable}

In the case where the class $\mathscr C$ is the class of all graphs excluding $K_k$ as a minor, we derive (see \Cref{th:pd-KF}) that the maximum girth $g(K_k,p)$ of a $K_k$-minor free graph that is not $p$-path degenerate is bounded by
     $$\Bigl(\frac{8}{3}\log_{2}k+c\Bigr)(p-1)\le g(K_k,p)<(4\log_{2}k+2\log_{2}\log_{2}k+O(1))(p-1).$$

We conjecture that $g(K_k,p)$ is close to the upper bound:
\begin{restatable}{conjecture}{ConjKkp}
We have
 $$g(K_k,p)=(4\log_2 k+O(1))(p-1).$$
\end{restatable}

Note that \Cref{conj:cage} would imply that, for some constant $c$, we have
$g(K_k,p)\geq (4\log_2 k+c)(p-1)$ for all $p$ and infinitely many $k$.

\subsection{Applications}
We review here some of the applications of our study of path-degeneracy.

\subsubsection*{Generalized $r$-arboricity}
Given an integer $r\ge1$, 
the \emph{generalized $r$-arboricity} $\arb_r(G)$ of a graph $G$ \cite{NOZ2014}
    is the minimum number of colors in a coloration of the edges of $G$ such that every cycle $C$ of $G$ gets  at least $\min\{|C|,r+1\}$ colors.
    Note that if $r=1$, then $\arb_1(G)$ is the usual arboricity $\arb(G)$.

    The generalized $r$-arboricity $\arb_r(G)$ is  unbounded for planar graphs: for any $r\ge2$,  there are planar graphs with arbitrarily large $\arb_r(G)$ (see \cite{BBCFGM2019} for such examples).
However, for every  class  with bounded expansion,  $\arb_r(G)$ is  bounded  for graphs with sufficiently large girth (depending on $r$) \cite{NOZ2014}.
One can deduce from the definition of $\arb_r(G)$ that  if a graph $G$ has girth at least $r+1$, then $\arb_r(G)\ge r+1$.
We prove that every  $(r+1)$-path degenerate graph  $r$ satisfies $\arb_r(G)\le r+1$ in \Cref{th:arb-PD}.

In this context, we focus on the bound for the generalized $r$-arboricity of large girth graphs in a given class.
We study the invariant $\A_{r}(\mathscr{C})$ of a  class $\mathscr{C}$ with bounded expansion,
 which  was defined 
in \cite{BBCFGM2019} as
\[\A_r(\mathscr C):=\min\bigl\{k: \exists g \ \forall G\in \mathscr{C} \  (\girth(G)\ge g)\Rightarrow (\arb_{r}(G)\le k)\bigr\}.\]

Obviously,
the best possible for 
 $\A_{r}(\mathscr{C})$ is  $r+1$. %
 Bartnicki et al. 
  \cite{BBCFGM2019} 
  proved  that $\A_{r}(\mathscr{C})=r+1$ holds for planar graphs and more generally, for graphs with any fixed genus. They  also indicated that this equality  does not hold for every class with bounded expansion and  conjectured the following.
  
\begin{restatable}[{\cite[Conjecture 7]{BBCFGM2019}}]{conjecture}{ConjA}
\label{conj:A-MC}
    Every   proper minor-closed class of graphs
 $\mathscr{C}$  satisfies   $\A_{r}(\mathscr{C})=r+1$.
\end{restatable}

We  confirm this conjecture 
by proving the following theorem in a strong form.

\begin{restatable}{theorem}{ThmAr}\label{th:arb-subexponential}
Let $\mathscr C$ be a class with sub-exponential expansion and let $r$ be a positive integer.

If $\mathscr C$ contains some graph $G$ with  $r<\girth(G)<\infty$,
then  $\A_{r}(\mathscr{C})=r+1$; otherwise,  $\A_{r}(\mathscr{C})=1$.
\end{restatable}

\subsubsection*{Generalized $r$-acyclic chromatic index}
Given an integer $r\ge 3$, 
the \emph{generalized $r$-acyclic chromatic index} $\ap_r(G)$ of a graph $G$ \cite{GGW2006} is the minimum number of colors in a proper coloration of the edges of $G$ such that every cycle $C$ of $G$ gets  at least $\min\{|C|,r\}$ colors. (Thus, it is a strengthening of generalized $(r-1)$-arboricity requiring the edge coloration to be proper.)
    The case $r=3$ is the well-known acyclic edge coloring.

       Obviously, if  a graph $G$ has
 maximum degree $\Delta$ and girth at least $r$, then
             at least $\max\{\Delta,r\}$ colors are required for its  generalized $r$-acyclic edge colorings.
 We prove that every  $(r+1)$-path degenerate graph  $G$  satisfies $\ap_r(G)\le \max\{\Delta,r\}$ in \Cref{th:acyclic-pd}.
    
 In general, graphs with large girth can have arbitrary large generalized $r$-acyclic chromatic indices,
for instance, Greenhill and Pikhurko \cite{GP2005} proved that for  fixed integer $r\ge4$, there exist graphs $G$ such that 
 $\ap_r(G)\ge c\Delta^{\lfloor r/2 \rfloor}$ for some  positive constants $c$. 
However, if restricted to some  nice classes of graphs,    large girth graphs may be forced to have small generalized $r$-acyclic chromatic indices, in the sence that they are close to the best possible  value of $\max\{\Delta,r\}$. 
For example, combining \cite[Theorem 3.3]{HWLL2009} 
and \cite[Theorem 6]{ZWYLL2012} together, 
we obtain
for  every  integer  $r\ge3$, every planar graph $G$ with  $\Delta\ge3$ and $\girth(G)\ge5r + 1$ satisfies $\ap_r(G)=\max\{\Delta,r\}$.

In this context, we extend the existence of a linear bound for the girth ensuring $\ap_{r}(G)\leq \max\{\Delta,r\}$ to all proper minor-closed classes in \Cref{cor:acyclic-PMCC}.
A more general form, for the class with sub-exponential expansion, is given as follows.

  \begin{restatable}{theorem}{ThmAcyclic}\label{th:acyclic-subexponetial}
Let $r\ge3$ be an integer, let $\mathscr{C}$ be a  class with sub-exponential class, and let $g(p)$ be the girth ensuring $p$-path degeneracy in $\mathscr C$.

Then, every graph $G\in \mathscr{C}$ with maximum degree $\Delta$ and girth at least $g(r+1)$ that is not a forest satisfies $\ap_{r}(G)= \max\{\Delta,r\}$.
\end{restatable}

\subsubsection*{Weak coloring number} Given an integer $r\ge1$,  the  \emph{weak $r$-coloring number} $\wcol_r(G)$ of a graph $G$ \cite{KY2003} is a minimum integer $k$ such that there exists a linear order $\pi$ on the vertices of $G$, so that for every vertex $v\in V(G)$, there are at most $k$ vertices $u\le_{\pi} v$ (possible $u=v$) that can be reached from $v$ by a path $P$ of length at most $r$, for which every internal vertex $w$ of $P$ satisfies $u< _{\pi} w$.

The growth rate of  weak $r$-coloring number for graphs in a fixed proper minor-closed class $\mathscr{C}$ has been extensively studied, see \cite[Table 1]{HLMR2025} for a summary on the lower and upper bounds on  $\max_{G\in\mathscr{C}} \wcol_r(G)$.
In particular, for the class  of all $K_k$-minor free graphs ($k\ge4$), the lower and upper bounds on  $\max_{G\in\mathscr{C}} \wcol_r(G)$ are 
$\Omega(r^{t-2})$ \cite{GKRSS2018} and $O(r^{t-1})$ \cite{VOQRS2017}, respectively.
Even for the class of planar graphs $\mathscr{C}$, the exact  growth rate of maximum weak $r$-coloring numbers
 is still unknown, and 
the known lower and upper bounds on $\max_{G\in\mathscr{C}} \wcol_r(G)$ are $\Omega(r^2\log r)$ \cite{JM2021} and $O(r^3)$ \cite{VOQRS2017}, respectively.

 In this context,  we first establish the path degeneracy conditions that ensure linear weak $r$-coloring numbers for general graphs in \Cref{th:wocl-pd}.
 Then, based on this theorem, 
 we prove that 
for  large girth graphs in any proper minor-closed class $\mathscr{C}$ and more generally, in the class with sub-exponential expansion,  
the weak $r$-coloring numbers grow linearly with $r$.

  \begin{restatable}{theorem}{ThmWcol}\label{th:wcol-subexponential}
Let $r\ge 1$ be an integer, let $\mathscr{C}$ be a  class with sub-exponential class, and let $g(p)$ be the girth ensuring $p$-path degeneracy in $\mathscr C$.

Then, for every graph $G\in  \mathscr{C}$, we have

\[
\wcol_r(G)\leq\begin{cases}
r+2+\lfloor\log_2\bigl(\frac{q-1}{q-r}\big)\rfloor,&\text{if }\girth(G)\ge g(2q) \text{ and }r<q<2r,\\
r+2,&\text{if }\girth(G)\ge g(4r).
\end{cases}
\]
\end{restatable}

\subsection{Structure of the paper}

The paper is organized as follows:
In \Cref{Sec:Pre}, we recall some needed   preliminaries and notions. 
In \Cref{sec:pd}, we study the path degeneracy with graphs. In particular, we relate the path degeneracy and shallow (topological) minor, and further characterize the girth condition to ensure $p$-path degeneracy for some nice classes of graphs. 
Finally, in \Cref{sec:app}, we  discuss connections between path degeneracy and three different coloring problems:
generalized arboricity (in \Cref{sec:arboricity}),
general acyclic chromatic indices
(in \Cref{sec:acyclic}), 
and generalized coloring numbers
(in \Cref{sec:weak}).

\bigskip

 \section{Preliminaries}\label{Sec:Pre}

 Let $G=(V(G),E(G))$ be a graph.
 The \emph{average degree} $\avd(G)$ is the average of the degrees of $G$, hence $2\|G\|/|G|$ (if $G$ is not empty, and $0$ by convention if $G$ is empty). 
The  \emph{maximum average degree} of a graph $G$, denoted by $\mad(G)$ is defined as the maximum of the average degrees of the  subgraphs of $G$: $\mad(G)=\max\{\avd(H)\colon H\subseteq G\}$.
The  maximum average degree of a class of  graphs $\mathscr{C}$ is defined by %
$\mad(\mathscr{C})=\sup\{\mad(G)\colon G\in \mathscr{C}\}$.

A graph $H$ is a \emph{minor} of $G$ if $H$ can be obtained from  $G$ by means of a sequence of vertex and edge deletions and edge contractions.
 A class of graphs $\mathscr{C}$ is \emph{minor-closed} if for every graph $G\in \mathscr{C}$, any minor of $G$ also belongs to $\mathscr{C}$;
a minor-closed class $\mathscr{C}$ is \emph{proper} if it does not contain all finite graphs.
The class of planar graphs is a traditional example of the proper minor-closed classes.

Let $r$ be a nonnegative  half-integer (that is, $r\in \frac{1}{2}\mathbb{N}$). A graph $H$ is a \emph{shallow topological minor of $G$ at depth $r$} if a $\le(2r)$-subdivision of $H$ is a subgraph of $G$.  More generally,
a graph $H$ is a \emph{shallow minor of $G$ at depth $r$}, if  each vertex $v\in V(H)$ corresponds to a rooted tree $(T_v,r_v)$ such that 
\begin{itemize}
\item for every vertex $v\in V(H)$,   $T_v$   has  radius at most $\lceil r \rceil$, where the radius of $(T_v,r_v)$ is the maximum length of a path of $T_v$ with one endpoint $r_v$,
    \item  the graphs $T_v$ (for  $v\in V(H)$) are vertex-disjoint subgraphs of $G$,  
\item for every edge $uv\in E(H)$, there exists an edge $e$ between $T_u$ and $T_v$ such that the length of the path between  $r_u$ and $r_v$ in $T_u\cup T_v \cup \{e\}$ is at most $2r+1$.
\end{itemize}

We denote by $G\mathbin{\nabla}r$ the class of all shallow minors of $G$ at depth $r$, and denote by  $\nabla_{r}(G)$ the maximum density of a graph in $G\mathbin{\nabla}r$, that is:
$$\nabla_{r}(G)=\max\biggl\{\frac{\|H\|}{|H|}:H\in G\mathbin{\nabla}r \biggr\}.$$

The \emph{maximum density} of a minor of  $G$, denoted by $\nabla(G)$, is defined as  $$\nabla(G)=\max\left\{\nabla_{r}(G):r\in \mathbb{N}\right\}.$$

\begin{proposition}[{\cite[Proposition 4.1]{Sparsity}}]
\label{prop:arithmetic}
    Let $a,b,c\in  \frac{1}{2}\mathbb{N}$.
    If $(2c+1)\geq (2a+1)(2b+1)$, then for every graph $G$, it holds that
$$(G\mathbin{\nabla}a)\mathbin{\nabla}b\subseteq G\mathbin{\nabla}c.$$
\end{proposition}

The \emph{expansion} of a class $\mathscr{C}$ is the function ${\rm Exp}_{\mathscr C}:\frac{1}{2}\mathbb N\rightarrow [0,+\infty]$ defined by
\[
{\rm Exp}_{\mathscr C}(r)=\sup\{\nabla_r(G)\colon G\in\mathscr C\}.
\]
A class $\mathscr C$ has \emph{bounded expansion} (resp. \emph{sub-exponential expansion}, \emph{polynomial expansion}) if ${\rm Exp}_{\mathscr C}(r)$ is finite for every half-integer $r$ (resp. ${\rm Exp}_{\mathscr C}(r)=2^{o(r)}$, ${\rm Exp}_{\mathscr C}(r)=r^{O(1)}$). %
Note that a class $\mathscr C$ is included in a proper minor-closed class if and only if ${\rm Exp}_{\mathscr C}(r)=O(1)$.

\section{Path degeneracy}
\label{sec:pd}
In  this section, we  consider bounds on  the girth  that ensures $p$-path degeneracy for graphs in  a fixed class. After two preliminary subsections, 
we will discuss the considered classes from the more general to the more specific:  classes with sub-exponential expansion (\Cref{sec:subexp}),  classes with polynomial expansion (\Cref{sec:pol}), and eventually proper minor-closed classes (\cref{sec:minor}).

\subsection{Definition and basic facts}
\label{sec:def}
A \emph{strict ear} of a graph $G$ is a path of $G$, whose internal vertices have degree $2$ in $G$ and whose endpoints are distinct.
Given an integer $p\ge1$, a 
 \emph{$p$-reduction} of a graph $G$ is the deletion of
\begin{itemize}
    \item either an isolated vertex,
    \item or a vertex with degree $1$,
    \item or the internal vertices of a strict ear of $G$ with length at least $p$.
\end{itemize}
\begin{remark}
\label{rem:exact}
In the third $p$-reduction, we can require the strict ear to have length exactly $p$. Indeed, if a strict ear $P=(u,w_1,\dots,w_{q-1},z$ has length $q>p$, we can first delete the strict ear $(u,w_1,\dots,w_p)$, then the remaining vertices $w_{p+1},\dots,w_{q-1}$ in order (as they will have degree $1$).
\end{remark}
 A graph $G$ is \emph{$p$-path degenerate} if it can be reduced to the empty graph by a sequence of $p$-reductions.
 As every graph is obviously $1$-path degenerate, we only consider $p$-path degeneracy for 
 $p\ge2$.

A graph $G$ is \emph{$p$-reducible} if there exists a  $p$-reduction that can be performed on $G$; otherwise, $G$ is called \emph{$p$-irreducible}.
It is clear that every $p$-irreducible graph has minimum degree at least 2.
A  \emph{minimal $p$-irreducible} graph $G$ is a $p$-irreducible graph such that every proper subgraph of $G$ is $p$-reducible.

The following useful fact is easy to establish.

\begin{fact}
\label{fact:hered}
If $H$ is a subgraph of a $p$-path degenerate graph $G$, then $H$ is $p$-path degenerate.
\end{fact}

Note that,   a graph is $p$-path degenerate if and only if all its connected components are $p$-path degenerate. However, the $p$-path degeneracy of a graph cannot be reduced to its  2-connected components. 
For an instance, 
see \Cref{fig:theta-graph},  where  all blocks of the graph  are $p$-path degenerate, but the graph itself is not $p$-path degenerate.
(In the figure, $\Theta(l_1,l_2,\ldots,l_t)$ denotes 
the \emph{generalized theta graph},  which is the graph consisting of $t$  internally disjoint paths with lengths  $l_1\le l_2\le \ldots \le l_t$ connecting two vertices.)

 \begin{figure}[htbp]  
	\centering	
	\begin{minipage}[t]{6.3cm}
		\centering
		\resizebox{5.5cm}{3.5cm}{\includegraphics{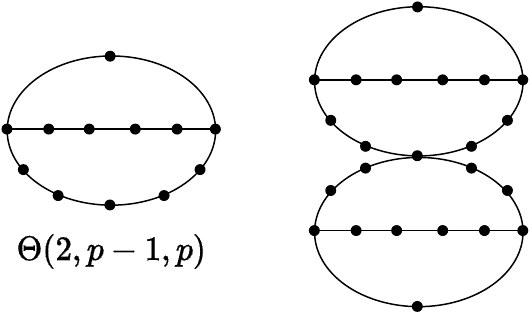}}
	\caption{A $p$-reducible graph and a $p$-irreducible graph}
	\label{fig:theta-graph}
	\end{minipage}
	\begin{minipage}[t]{6cm}
		\centering
		\resizebox{4.1cm}{3.3cm}{\includegraphics{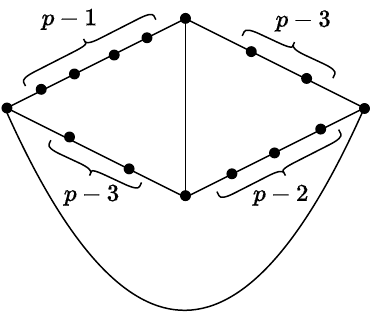}}
	\caption{A subdivision of the complete graph $K_4$.}
	\label{fig:2p-2}
	\end{minipage}
\end{figure}

We now relate the path degeneracy for graphs with large girth to  shallow topological minors.
The next lemma is a useful tool in this section.

\begin{lemma}\label{lemma:2p-1}
Let $G$ be a graph with  girth at least $2p-1$. Then, $G$ is $p$-path degenerate if and only if it does not contain a  subgraph $H$ that is a $\leq (p-2)$-subdivision of a simple graph with minimum degree at least $3$ (that is, it contains no  topological minors of depth $\frac{p-2}{2}$ with minimum degree at least 3).
\end{lemma}
\begin{proof}
    Assume $G$ is not $p$-path degenerate. Let $H$ be a minimal  subgraph of $G$ that is $p$-irreducible. Obviously, $H$ is connected and has minimum degree at least 2.
Therefore, $H$  is either a cycle with length at most $p$, or a subdivision of a multigraph (possibly with loops) $H'$ with $\delta(H')\ge 3$.
   Because  $\girth(H)\ge\girth(G)\ge2p-1$, $H$ must be  a subdivision of a multigraph $H'$ with $\delta(H')\ge 3$.
   Since $H$ is $p$-irreducible, $H$ is a  $\leq (p-2)$-subdivision of a simple graph $H'$ with  minimum degree at least $3$.

    Conversely, assume that $G$ contains a subgraph $H$ that is a $\leq (p-2)$-subdivision of a simple graph  $H'$ with $\delta(H')\ge 3$.
    Then, $H$  is a $p$-irreducible subgraph of $G$.
    Therefore, according to \Cref{fact:hered},  $G$ is not $p$-path degenerate. The lemma is proved.
\end{proof}

Remark that the lower bound  $2p-1$ on the girth in \Cref{lemma:2p-1} is tight. For an example, see 
\Cref{fig:2p-2}, where the graph has girth $2p-2$, and  is neither $p$-path degenerate nor contains
 a %
 subgraph  that is a $\leq (p-2)$-subdivision of a simple graph with minimum degree at least $3$. 

\subsection{Dense shallow minors of graphs of large girth}

The following two lemmas play an essential role in the proofs in this section. They are slightly different from the original results stated in \cite{DR2004,KO2003}, but can be simply derived from their respective proofs.

We first look at 
 {\cite[Theorem 1]{DR2004}} proved by Diestel and  Rempel.

\begin{theorem*}
\label{th:DR-original}
For any integer $k$, every graph $G$ of girth $g(G) > 6\log k + 3$
and $\delta(G)\geq3$  has a minor $H$ with $\delta(H)> k$.
\end{theorem*}

Diestel and  Rempel proved this theorem by  constructing a  partition of  the vertex set of $G$, where every part induces a subgraph with radius at most  $2\lfloor \log k\rfloor$, and every part is adjacent to at least $2^{\lfloor \log k\rfloor+1}$  other parts. Obviously, contracting each part of this partition  gives a shallow minor  $H$ of $G$ at depth $2\lfloor \log k\rfloor$, and $\delta(H)\ge 2^{\lfloor \log k\rfloor+1}>k$. 
Let $r=\lfloor \log k\rfloor$,
we  have $\nabla_{2r}(G)\geq \frac12\cdot 2^{r+1} =2^{r}$.
Hence, the following lemma follows immediately.

\begin{lemma}[{\cite[Proof of Theorem 1]{DR2004}}]
\label{th:DR}
	Let $r$ be an integer and let $G$ be a graph. If $\girth(G)>6r+3$ and $\delta(G)\geq3$ then $\nabla_{2r}(G)\geq2^r$. \hfill\qed
\end{lemma}

In \cite[Theorem 1]{KO2003}, K\"uhn and Osthus prove the following result.
\begin{theorem*}
Let
$k\geq 1$ and $r\geq 3$ be integers and put $g: = 4k+ 3$. Then every graph $G$ of minimum degree $r$ and girth at least $g$ contains a minor of minimum degree at least $(r-1)^{k+1}/48=(r-1)^{(g+1)/4}/48.$
\end{theorem*}

A closer look at the proof of this theorem in \cite{KO2003}
shows that a minor $M$ with depth $k$ is constructed, which has average degree at least $(r-1)^{k+1}/24$ and from which a minor with minimum degree $(r-1)^{k+1}/48$ is extracted. The minor $M$ itself witnesses $\nabla_k(G)\geq  \frac12\cdot\frac{(r-1)^{k+1}}{24}\geq
\frac12\cdot\frac{2^{k+1}}{24}=\frac{2^k}{24}$.
Hence, we deduce the following.

\begin{lemma}[{\cite[Proof of Theorem 1]{KO2003}}]
\label{th:KO}
	Let $r\ge 1$ be an integer and let $G$ be a graph. If $\girth(G)\ge4r+3$ and $\delta(G)\geq3$ then $\nabla_{r}(G)\geq\frac{2^r}{24}$.\hfill\qed
\end{lemma}

\subsection{Classes with sub-exponential expansion}
\label{sec:subexp}
We start with the weakest assumption on the expansion of a class $\mathscr C$ allowing to prove that high girth graphs in $\mathscr C$ have path degeneracy. %
This basic result will serve as a warm-up for the much less trivial cases, where we shall seek (almost tight) explicit bounds.

\ThmSE*
\begin{proof}
    Let $r\geq p$  ($\ge2$) be an integer such that $\nabla_{3pr}(G)<2^{r}$ holds for every $G\in\mathscr C$. (Such an integer exists as the class has sub-exponential expansion.)
    Let $g_p=(6pr+3)(p-1)$, and assume for contradiction that there exists a graph $G\in\mathscr C$ with girth at least $g_p$ that is not $p$-path degenerate.

    According to \Cref{lemma:2p-1}, there exists $G'\in G\mathbin{\nabla}\frac{p-1}{2}$ with $\delta(G')\geq 3$ (and $\girth(G')\geq 6pr+3>6r+3$).
    According to \Cref{prop:arithmetic}, $r\geq p$ implies $G'\mathbin{\nabla}2r\subseteq G\mathbin{\nabla}3pr$. Hence,
    $\nabla_{3pr}(G)\geq \nabla_{2r}(G')$.
    However, according to  \Cref{th:DR}, $\nabla_{2r}(G')
\geq 2^{r}$,
 contradicting the assumption  $\nabla_{3pr}(G)<2^r$.
\end{proof}

\subsection{Classes with polynomial expansion}
\label{sec:pol}
Classes with polynomial expansion  have recently attracted a lot of attention, because they not only include proper minor-closed classes, but also a large diversity of geometrically defined classes of graphs.
(In particular, every class of intersection graphs of  objects (of $\mathbb R^d$) in a family with low density has polynomial expansion \cite{HQ2017}.)
Moreover, these classes are exactly those have strongly sublinear vertex separators \cite{DN2016}, giving them very interesting algorithmic properties.

Before we proceed to our analysis of classes with polynomial expansion,  we take time to 
prove some technical lemma related to the inequality $A\log x+B<x$.

Recall the Lambert $\W[]$ function.
Let $\W[0]$ (resp. $\W$) be principal branch (resp. the branch with index $-1$) of Lambert $\W$ function, which are the branches of the converse relation of the function $z\mapsto ze^z$.

 \begin{lemma}
 \label{lem:sol2}
 Let $A,B>0$ be reals, and let
 $\beta\geq 0$ be minimum such that the inequation
    $x> A \log x + B$
is satisfied for all $x>\beta$.

If $B\leq A(1-\log A)$, then $\beta=0$; otherwise, $\beta=-A\W(-e^{-B/A}/A)$.
\end{lemma}
\begin{proof}
This inequation rewrites as $y> e^{B/A}A\log y$, where $y=e^{B/A}x$, that is $\log y/y< e^{-B/A}/A$.
If $Ae^{B/A}\leq e$, i.e. $B\leq A(1-\log A)$, this is always satisfied as the maximum of  $\log y/y$ is $1/e$.

Hence, we can assume $B> A(1-\log A)$.
Let $\alpha=Ae^{B/A}$.
The equation $\log y/y=e^{-B/A}/A$
     rewrites as
    $z e^z=-1/\alpha$, where $y=e^{-z}$. The equation $z e^z=-1/\alpha$ has two solutions, namely $\W[0](-1/\alpha)$ and $\W(-1/\alpha)$.
    As (for every complex $t$) $e^{-\W[0](t)}=\W[0](t)/t$ and $e^{-\W(t)}=\W(t)/t$ we deduce that we have two solutions for $y$. Precisely,
    as $y$ increases from $0$ to $+\infty$, the value of $\log y/y$ first increases, reaches $1/\alpha$ for $y=-\alpha \W[0](-1/\alpha)$, then reaches $1/e$ at  $y=e$, then decreases, reaches $1/\alpha$ for $x=-\alpha \W(-1/\alpha)$ and then continues to decrease as $y\rightarrow+\infty$. Thus, the threshold value for $y$  is 
    $-\alpha \W(-1/\alpha)$.
    
    Thus, we get 
$\beta=-e^{-B/A}\alpha\W(-1/\alpha)=-A\W(-e^{-B/A}/A)$.
\end{proof}   

The following approximation will be useful to have some explicit bounds based on standard functions.
\begin{lemma}[\cite{6559999}]
\label{thm:W}
    For $u>0$, we have
    \begin{equation*}
        \label{eq:Wapprox}
-1-\sqrt{2u}-u<\W(-e^{-u-1})<-1-\sqrt{2u}-\frac{2}{3}u.
    \end{equation*}
\end{lemma}

\ThmPol*

\begin{proof}
Let $A=\frac{b}{\log 2}$, $C=(24\sqrt{2}a)^{1/b}$, $\gamma(p)=2\left\lfloor-2A\W\left(-\frac1{ACp}\right)\right\rfloor+4$, and  $g_p=\max(7,\gamma(p))$. Assume for contradiction that there exists    $G\in\mathscr C$ with $\girth(G)>g_p(p-1)$ that is not $p$-path degenerate.

  According to \Cref{lemma:2p-1},  there exists $G'\in G\mathbin{\nabla}\frac{p-1}{2}$ with $\delta(G')\geq 3$ and $\girth(G')>g_p\geq 7$.

    Let $r=\Bigl\lfloor \frac{\girth(G')-3}{4} \Bigr\rfloor\geq 1$ (hence $4r+7>\girth(G')\geq 4r+3$).
According to \Cref{th:KO}, as  $\delta(G')\ge3$, we have $\nabla_{r}(G')\geq\frac{2^r}{24}$.
This, together with \Cref{prop:arithmetic}, implies that
\begin{align*}
r&\le \log_{2} (24\nabla_{r}(G'))
\leq  \log_{2} (24\nabla_{\frac{2pr+p-1}{2}}(G))\\
&\leq  \log_{2} \biggl(24 a\Bigl(p\bigl(r+\frac12\bigr)\Bigr)^b\biggr)
=\log_2(24ap^b)+\frac{b}{\log 2}\log\bigl(r+\frac12\bigr).
\intertext{Hence, if $r'=r+1/2$, we have}
r'&\le \frac{b}{\log 2}\log r'+\frac{\log(24\sqrt{2}ap^b)}{\log 2}=A\log r'+B,
\end{align*}
where $B=\frac{\log(24\sqrt{2}ap^b)}{\log 2}$. As
  $e^{\frac{B}A}=e^{\frac{\log(24\sqrt{2}ap^b)}{b}}=(24\sqrt{2}a)^{1/b}p=Cp$,
according to \Cref{lem:sol2} we have
\[
r'\leq -A\W\left(-\frac1{ACp}\right)
\]
As $2r'$ is an integer, we deduce
\[
2r'\leq \left\lfloor-2A\W\left(-\frac1{ACp}\right)\right\rfloor.
\]

It follows that 
\[
\girth(G')< 4r+7=4r'+5\leq 
2\left\lfloor-2A\W\left(-\frac1{ACp}\right)\right\rfloor+5.
\]
Hence, as $\girth(G')$ is an integer, 

\[
\girth(G')\leq 2\left\lfloor -2A\W\left(-\frac1{ACp}\right)
\right\rfloor+4\leq \gamma(p) \leq g_p,
\]
which contradicts $\girth(G')>g_p$.

Finally, according to \Cref{thm:W},
we have the uniform bound
\[\gamma(p)<4b\log_2p+4A\sqrt{2\log(ACp)-2}+4b\log_2(AC)+4. \]
Therefore,   $\gamma(p)\sim 4b\log_2p$ and the theorem holds.
\end{proof}

We now prove that the bound on the girth ensuring $p$-path degeneracy is tight up to a multiplicative factor.

\begin{lemma}
\label{lem:poly_lower}
   Let $\alpha>0$ be such that there is a constant $c$ such that for infinitely many integers
$g$ there exists a cubic graph of girth at least $g$ and order at most $c\,2^{\alpha g}$. 

    For every real $b>0$, let $\mathscr X_b$ be the class of all graphs $H$ such that, for all half-integers $r$ we have
\[
\nabla_r(H)\leq\max(
 \sqrt{c}/2,3)\cdot r^b+3e^{b}.
\]
Then, for infinitely many $p\in\mathbb N$ the class $\mathscr X_b$ contains a graph $H$ with
\[
\girth(H) \ge -\frac{2b}{\alpha\log 2}\W\left(-\frac{\log 2}{(p-1)b}\right)(p-1)\sim\frac{2}{\alpha}bp\log_2 p.
\]
that is not $p$-path degenerate.
\end{lemma}
\begin{proof}
We first give some bounds on the densities of the shallow minors of cubic graphs with large girths and small orders.
\begin{claim}\label{claim:lower}
Assume that $G$ is a cubic graph of girth at least $g$ and order at most $c\,2^{\alpha g}$, for some constant $\alpha>0$.
Then, $\nabla(G)<\frac{\sqrt{c}}{2}\,2^{\alpha g/2}+1$ and, for every integer~$r$, $\nabla_r(G)\leq 3\cdot 2^{r-1}$.
\end{claim}
\begin{clproof}
Let $H\in G\mathbin{\nabla} r$. Obviously,
$\nabla_r(G)\leq \frac12\, 3\cdot 2^{r}$.
Let $H$ be a minor of $G$. If a vertex $v$ of $H$ has degree $k$, it corresponds to connected subgraph of $G$ of order least $k-2$, as $G$ is cubic. Hence, if $H$ has $n$ vertices and $m$ edges, we have 
$|G|\geq \sum_{v\in V(H)}({\rm d}_H(v)-2)=2(m-n)> 2(m-n)\frac{2m}{n^2}>4(d-1)^2$, where $d=m/n$. 
It follows that $d<\sqrt{|G|}/2+1$.
Hence,  $\nabla(G)<\frac{\sqrt{c}}{2}\,2^{\alpha g/2}+1$. %
\end{clproof}

Define 
\[
p=\frac{2}{\alpha g}2^{\frac{\alpha g}{2b}}+1,
\]
so that 
\[
2^{\frac{\alpha g}{2}}=\Bigl(\frac{\alpha g}{2}\Bigr)^b\,(p-1)^b.
\]
Consider the $(p-2)$-subdivision $H$ of $G$. Note that $H$ is not $p$-path degenerate (as $G$ is cubic) and 
$\girth(H)=g(p-1)$.

Let $r$ be a half-integer.
Every minor of $H$ has depth $r$ can be obtained by first contracting
some paths with length at most $p-1$ into a single edge, then taking a minor at depth $\lceil \frac{r}{p-1}\rceil$.
Thus,
\[
\nabla_r(H)\leq\nabla_{\lceil \frac{r}{p-1}\rceil}(G)\leq \nabla_{\frac{r}{p-1}+1}(G)\leq 3\cdot 2^{\frac{r}{p-1}}.
\]

If $\frac{r}{p-1}\geq \frac{\alpha g}{2}$ we have
\[
\nabla_r(H)\leq \nabla(G)\leq \frac{\sqrt{c}}{2}2^{\alpha g/2}+1=\frac{\sqrt{c}}{2}\left(\frac{\alpha g}{2}\right)^b(p-1)^b+1\leq \frac{\sqrt{c}}{2} r^b+1.
\]

If $\frac{b}{\log 2}\leq \frac{r}{p-1}\leq \frac{\alpha g}{2}$, 
as $2^x/x^b$ is increasing for $x\geq b/\log 2$,  we have
\[
\frac{2^{\frac{r}{p-1}}}{r^b}=
\frac{1}{(p-1)^b}\cdot\frac{2^{\frac{r}{p-1}}}{\bigl(\frac{r}{p-1}\bigr)^b}\leq 
\frac{1}{(p-1)^b}\cdot\frac{2^{\frac{\alpha g}{2}}}{(\alpha g/2)^b}=1.
\]
Hence, 
$\nabla_r(H)\leq 3\cdot 2^{\frac{r}{p-1}}\leq 3r^b$.

Finally, if $ \frac{r}{p-1}<\frac{b}{\log 2}$ we have 
\[\nabla_r(H)\leq \nabla_{\frac{r}{p-1}+1}(G)\leq3\cdot \,2^{b/\log 2}=3e^{b}.\]

Hence, for every half-integer $r$ we have
\[
\nabla_r(H)\leq \max(
 \sqrt{c}/2,3)\cdot r^b+3e^{b},
\]
that is: $H\in\mathscr X_b$.

Finally, inverting $2^{\frac{\alpha g}{2}}=\left(\frac{\alpha g}{2}\right)^b\,(p-1)^b$, we get 
$g=-\frac{2b}{\alpha\log 2}\W\left(-\frac{\log 2}{(p-1)b}\right)$, from which we deduce the expression of $\girth(H)$.

\end{proof}

Cubic graphs of minimum order with girth at least $g$ are called
\emph{$g$-cages}. We consider the  general upper bound for the
order of $g$-cages given by Biggs \& Hoare \cite{biggs1983sextet} and Weiss \cite{weiss1984girths}.

\begin{lemma}[\cite{biggs1983sextet,weiss1984girths}]
\label{lem:cage}
There is a constant $c$ such that for infinitely many integers
$g$ there exists a cubic graph of girth at least $g$ and order at most $c\,2^{3g/4}$.
\end{lemma}

We can now provide a lower bound for the girth ensuring $p$-path degeneracy in classes with polynomial expansion.

\ThmPolLower*
\begin{proof}
This is a direct consequence of \Cref{lem:poly_lower,lem:cage}.
\end{proof}

We conjecture that the bound given in \Cref{th:pd-polynomial} is tight up to second order term.

\ConjPolLower*

As a support to this conjecture, we note that its validity would follow from~\Cref{conj:cage} proposed in \cite[p. 164]{B1978}, asserting that there exist cubic graphs with girth at least $g$ and order at most $c2^{g/2}$. 
Indeed, if this later conjecture is true, then according to  \Cref{lem:poly_lower},  the lower bound on the girth in \Cref{thm:poly_lower} will be
\[
-\frac{4b}{\log 2}\W\left(-\frac{\log 2}{(p-1)b}\right)(p-2)\sim 4bp\log_2 p.
\]

\subsection{Minor-closed classes of graphs}
\label{sec:minor}
We now consider proper minor-closed classes of graphs, and provide a fine analysis of the bounds, either based on the maximum average degree of the graphs in the class or on the  maximum order of a clique in the class.

\ThmMC*
\begin{proof}
    Let $G$ be a  counterexample  of minimum order, that is, a graph $G$  in $\mathscr{C}$ with   $\girth(G)>(4\log_{2}d+2\log_{2}(\min\{d,576\})+3)(p-1)$, not  $p$-path degenerate, and with $|V (G)|$ minimum for this property.
According to Lemma~\ref{lemma:2p-1} and the minimality of $G$,  $G$  is a $\leq (p-2)$-subdivision of a simple graph $G'$ with minimum degree at least $3$. In other words,    $G'\in \mathscr{C}\mathbin{\nabla}\frac{p-1}{2}$.
Since  $\girth(G)>(4\log_{2}d+2\log_{2}(\min\{d,576\})+3)(p-1)$, we have $\girth(G')>4\log_{2}d+2\log_{2}(\min\{d,576\})+3$.

If $d\le576$, then $\girth(G')>6\log_{2}d+3$.
Let $r=\Bigl\lceil \frac{\girth(G')-3}{6} \Bigr\rceil-1$ (hence, $\girth(G')>6r+3$).
According to Lemma \ref{th:DR}, as  $\delta(G')\ge3$, we have $\nabla_{2r}(G')\geq2^r$.
Since $\mathscr{C}$ is a minor-closed class of graphs,
 it holds that 
$$\biggl\lceil \frac{\girth(G')-3}{6}\biggr\rceil-1=r\le \log_{2} \nabla_{2r}(G')\leq  \log_{2} \nabla(G)\le \log_{2}\frac{d}{2}= \log_{2}d-1.$$
Thus, we have 
$\girth(G') \le 6\log_{2}d+3$,
contradicting the inequality  $\girth(G')>6\log_{2}d+3$, which we derived from our assumptions.

Otherwise,  $d>576$. Then, we have $\girth(G')>4\log_{2}(24d)+3$.
Let $r=\Bigl\lceil \frac{\girth(G')-3}{4} \Bigr\rceil-1$ (hence, $\girth(G')>4r+3$).
According to Lemma \ref{th:KO}, as  $\delta(G')\ge3$, we have $\nabla_{r}(G')\geq\frac{2^r}{24}$.
Since $\mathscr{C}$ is a minor-closed class of graphs, there is
$$\biggl\lceil \frac{\girth(G')-3}{4}\biggr\rceil-1=r\le \log_{2} (24\nabla_{r}(G'))\leq  \log_{2} (24\nabla(G))\le \log_{2}\frac{24d}{2}.
$$
Therefore, we have 
$\girth(G') \le 4\log_{2}(24d)+3$,
a contradiction to the inequality  $\girth(G')>4\log_{2}(24d)+3$.
In both cases, no counterexample exist and the theorem holds. 
\end{proof}

We proceed to provide a lower bound for the girth ensuring $p$-path degeneracy in proper minor-closed  classes.

\begin{lemma}
\label{lem:MClower}
    Assume that $G$ is a cubic graph of girth at least $g$ and order at most $c\,2^{\alpha g}$, for some constant $\alpha>0$.
    Let $d=\sqrt{c}\,2^{\alpha g}+2$ and let $\mathscr C_d=\{H\colon \nabla(H)< d/2\}$.
    
    Then, $\mathscr C_d$ is a proper minor closed class with maximum average degree at most $d$, and for every integer $p\geq 2$, there exists a graph $G_p\in\mathscr C_d$ with
\[
    \girth(G_p)>\left(\frac{2}{\alpha}\log_2(d-2)-\frac{\log_2 c}{\alpha}\right)(p-1)
\]
that is not $p$-path degenerate.
\end{lemma}
\begin{proof}
According to \Cref{claim:lower}, $\nabla(G)<\frac{\sqrt{c}}{2}\,2^{\alpha g}+1=d/2$. 
 Let $p\geq 2$ be an integer. The $(p-2)$-subdivision $G_p$ of $G$ has girth at least $g(p-1)$ and satisfies $\nabla(G_p)=\nabla(G)<d/2$ (as $\nabla(G)>1)$.
 Thus, $G_p\in\mathscr C_d$ and, $G_p$ is  not $p$-path degenerate.
 The bound on the girth of $G_p$ then follows from the equality $d=\sqrt{c}\,2^{\alpha g}+2$.
\end{proof}

\ThmMClower*
\begin{proof}
Let $d_0>0$ be a real. According to~\cref{lem:cage} there exist a constant $c$ and 
infinitely many integers $g$ for which there exists a graph $G$ with minimum degree at least $3$, girth at least $g$, and order at most $c\,2^{3g/4}$.
We choose $g$ such that 
$d=\sqrt{c}\,2^{\alpha g}+2$ is greater than than $d_0$.
Then, the result follows from \Cref{lem:MClower}.
\end{proof}

We conjecture that the bound given in \Cref{th:pd-MC} is essentially tight. 
\ConjMClower*
Note that, according to \Cref{lem:MClower},   \Cref{conj:MC_lower} would follow from \Cref{conj:cage}. 
\bigskip

We now turn to  the  class  $\mathscr{C}$ of all graphs excluding a complete minor $K_k$ ($k\ge3$),
We define 
\[
g(K_k,p)=\sup\{g\colon \exists G\in\mathscr{C}, {\rm girth}(G)\geq g\text{ and $G$ is not $p$-path degenerate}\}.
\]

The  values of $g(K_k,p)$ for $k\le 4$ are kind of special. Let us simply mention without proof that
$g(K_3,p)=\infty$, $g(K_4,2)=\infty$, and $g(K_4,p)=2(p-1)$ for $p\geq 3$.

\begin{lemma}\label{lemma:Kk-pd}
      Let $\mathscr{C}$ be the class of all $K_k$-minor free graphs  $(k\ge5)$.
Then,   for any integer $p\ge2$, 
 $$g(K_k,p)=\max\{\girth(H):H\in \mathscr{C} \ \text{and}\  \delta(H)\ge3\}(p-1).$$
\end{lemma}
\begin{proof}
Since $K_4\in\mathscr C$, then
there exist  in $\mathscr{C}$ a graph with minimum degree at least $3$.
Let $H$ be such a graph with maximum possible girth $g$.
Hence, $g=\max\{\girth(H):H\in \mathscr{C} \ \text{and}\  \delta(H)\ge3\}$.
The $(p-2)$-subdivision 
of  $H$  belongs to $\mathscr{C}$,  has girth $\girth(H)(p-1)$, and is not $p$-path degenerate.  Hence, 
 $g(K_k,p)\ge g(p-1)$.

Conversely,
let $G$ be a graph with girth $g(K_k,p)$ that is not $p$-path degenerate.
Since $g(\mathscr{C},p)\ge g(p-1)>2p-1$, according to \Cref{lemma:2p-1}, 
$G$ contains, as an induced subgraph, a $\leq (p-2)$-subdivision of a graph $H'$ with minimum degree at least $3$. 
Since $\mathscr{C}$ is minor-closed,  $H'\in \mathscr{C}$ and $\girth(H')\le g$. Thus, we have
$g(p-1)\ge\girth(H')(p-1)\ge\girth(G)=g(\mathscr{C},p)$.   
\end{proof}

As an application, we get a tight bound for planar graphs. 
\begin{theorem}\label{th:pd-planar}
Every planar graph with girth at least $5p-4$ is $p$-path degenerate and there exists a planar graph with girth $5p-5$ that is not $p$-path degenerate.
\end{theorem}
\begin{proof}
It is a easy consequence of Euler’s Formula that every planar graph with minimum degree at least $3$ has girth at most $5$.
Note that every planar graph is $K_5$-minor free,
according to
\cref{lemma:Kk-pd}, every planar graph with girth at least $5p-4$ is $p$-path degenerate.

  For any integer  $p\ge2$, let $G_p$ denote the graph obtained from the dodecahedron by subdividing each edge exactly $p-2$ times.  Since the dodecahedron is a simple $3$-connected $3$-regular planar graph with girth $5$, the graph $G_p$ is  a planar graph with girth $5(p-1)$ that is not $p$-path degenerate.
\end{proof}

To derive the lower and upper bounds  on 
$g(K_k,p)$  for the more general case of $k$,
we need some ingredients.
The first one, is about existence of small graphs with minimum degree $3$ and large girth.

\begin{lemma}[{\cite[Proposition 2.2]{DR2004}}]
\label{prop:DR2004}
 There is a constant  $c\in \mathbb{R}$ such that for infinitely many  $k\in \mathbb{N}$ there exist cubic graphs with girth at least  $\frac{8}{3}\log_{2}k-c$ that have no minor with minimum degree greater than $k$. 
\end{lemma}

The second ingredient we shall use is the classical bound on the maximum number of edges of a $K_k$-minor free graph.
\begin{theorem}[\cite{Thomason2001}]
\label{Th:Thomason}
     There exists an explicit constant $\gamma=0.638\dots$ such that every graph with average degree at least $(\gamma +o(1))k\sqrt{\log_{2}k}$ contains a $K_k$ minor. 
\end{theorem}

Our result is the following theorem.
\begin{theorem}\label{th:pd-KF}
  Let $p\ge2$ be an integer. 
  Then for every integer $k$ we have
  \begin{align*}
  g(K_k,p)&<(4\log_{2}k+2\log_{2}\log_{2}k+O(1))(p-1).\\
  \intertext{
  Moreover, there is a constant  $c\in \mathbb{R}$ such that for infinitely many  $k\in \mathbb{N}$ we have}
    g(K_k,p)&\ge \Bigl(\frac{8}{3}\log_{2}k+c\Bigr)(p-1).
  \end{align*}
\end{theorem}
\begin{proof}
The first inequality is a direct consequence of \Cref{th:pd-MC,Th:Thomason}.

According to \Cref{prop:DR2004}, for infinitely many  $k\in \mathbb{N}$  there exists a cubic graph $H$ with $\girth(H)\ge \frac{8}{3}\log_{2}(k-2)-c'$ (for some constant $c'$) with no minor with minimum degree greater than $k-2$ (hence, no $K_k$-minor). By slightly changing the constant, we have
$\girth(H)\ge \frac{8}{3}\log_{2}k+c$.
The $(p-2)$-subdivision of $H$ is  $K_k$-minor free, has girth at least 
$(\frac{8}{3}\log_{2}k+c)(p-1)$, 
and is not $p$-path degenerate. 
\end{proof}

As observed by  Diestel and  Rempel \cite{DR2004},
if \Cref{conj:cage} is true, then for infinitely many integers $k$ there are cubic graphs with girth at least $4\log_2k+c$ (for a suitable constant $c$) that have no minor of minimum degree $k-1$ (and thus have no $K_k$ minor). 
Then, by \Cref{lemma:Kk-pd}, we will obtain $g(K_k,p)\ge(4\log_2k+c)(p-1)$.
Therefore, if  \Cref{lemma:Kk-pd} proven,
the constant 4 in the leading term of the upper bound on $g(K_k,p)$ in
\Cref{th:pd-KF}  can not  be replaced by a smaller one.
This motivates the next conjecture.

\ConjKkp*

\section{Applications}
\label{sec:app}

 \subsection{Generalized $r$-arboricity}
  \label{sec:arboricity}

In this subsection, we consider the generalized arboricity of graphs with large girth. We begin with a general statement relating generalized arboricity to path-degeneracy.

\begin{theorem}
 \label{th:arb-PD}
Let  $r\ge1$ be an integer and let $G$ be a $(r+1)$-path degenerate graph.
If $G$ is a forest, then $\arb_{r}(G)=1$; otherwise, $\arb_{r}(G)=r+1$. 
\end{theorem}
\begin{proof}
   If $G$ is a forest, obviously $\arb_{r}(G)=1$.
   Suppose that $G$ contains at least a cycle.
Since $G$ is $(r+1)$-path degenerate, 
$\girth(G)>r+1$. Hence, by the definition of $\arb_r(G)$, we have
  $\arb_{r}(G)\ge r+1$.

 We prove 
 by induction on the number of edges of $G$ that $\arb_r(G)\le r+1$.
  Note that the statement obviously holds if $G$ is a forest, and that we may assume that $G$ is connected. 
Since  $G$ is  $(r+1)$-path degenerate, there exists a $r$-reduction $R$ of $G$.
Let $H$ be the resulting graph after applying the reduction $R$ on $G$.
 By induction, there exists a coloring of the edges of $H$ using at most $r+1$ colors such that every cycle $C$ of $H$ gets at least $r+1$ colors. 
 We extend this coloring to $G$ as follows:
 if $R$ deletes a vertex $v$ with degree 1,  we assign the color $1$ to the edge $e$ incident to $v$; otherwise, $R$ deletes the internal vertices of a strict ear $P$ with length at least $r+1$,  we assign $r+1$ distinct colors to the edges of the path $P$.
 For the former, every cycle $C$ of $G$ is contained in $H$.
 For the latter, every cycle $C$ of $G$ is either contained in $H$ or includes all the edges of $P$.
 In either cases,  it  gets at least $r+1$ colors.
\end{proof}

Combining  \Cref{th:pd-polynomial} (resp. \Cref{th:pd-MC}) and  \Cref{th:arb-PD}, we immediately have the following two consequences.

\begin{corollary}
 \label{cor:arb-polynomial}
      Given an integer $r\ge1$.  Let $\mathscr C$ be a class with polynomial expansion, with ${\rm Exp}_{\mathscr C}(k)\leq a(k+\frac12)^b$ for some positive reals
 	$a,b$.
Let $G$ be a graph in $\mathscr C$ with
$$\girth(G)>\max\left(7,2\left\lfloor-\frac{2b}{\log 2}\W\left(-\frac{\log 2}{(24\sqrt{2}a)^{1/b}\,br}\right)\right\rfloor+4\right)p\sim 4br\log_2r.$$

If $G$ is a forest, then $\arb_{r}(G)=1$; otherwise, $\arb_{r}(G)=p+1$. 
\end{corollary}
\begin{corollary}\label{cor:arb-PMCC}
    Given an integer $r\ge1$.	Let $\mathscr{C}$ be a proper minor-closed class, let $d=\mad(\mathscr{C})$, and let
	 $G$ be a graph in $\mathscr{C}$ with girth
     greater than
     $$\girth(G)>(4\log_{2}d+2\log_{2}(\min\{d,576\})+3)r.$$

If $G$ is a forest, then $\arb_{r}(G)=1$; otherwise, $\arb_{r}(G)=r+1$. 
\end{corollary}

Remark that the bound $r+1$ on $\arb_r(G)$ was obtained for planar graphs with girth at least $2^{r+1}$ in \cite{BBCFGM2019}, further reduced to  girth at least $5r+1$ in \cite[Theorem 2]{CJV2020}. In \cref{cor:arb-PMCC}, we extended the existence of a linear bound for the girth ensuring $\arb_{r}(G)\leq r+1$ to all proper minor-closed classes. 
Note that, \cite[Theorem 2]{CJV2020} is also  a direct consequence of  \Cref{th:pd-planar,th:arb-PD} and  we restate it here.

\begin{theorem*}
    [{\cite[Theorem 2]{CJV2020}}]\label{th:arb-planar}
       Given  an integer $p\ge2$. Let $G$  be a planar graph with $\girth(G)\ge5p+1$ that is not a forest. Then $\arb_{p}(G)= p+1$.
\end{theorem*}

With \Cref{th:arb-PD} in hand, 
we can now settle \Cref{conj:A-MC}, which we recall for convenience.

\ConjA*

\ThmAr*
\begin{proof}
    If $\mathscr{C}$ contains some graph $G$ with $r<\girth(G)<\infty$, then by the definition of the generalized $r$-arboricity, every such graph $G$ satisfies
$\arb_r(G)\ge r+1$ and thus $\A(\mathscr{C})\ge r+1$. 
Since $\mathscr{C}$ has sub-exponential expansion,
according to \Cref{th:pd-subexponential},    there exists an integer $g_{r+1}$ such that every graph $G\in\mathscr C$ with girth at least $g_{r+1}$ is $(r+1)$-path degenerate. Then, according to  \Cref{th:arb-PD}, every graph $G\in\mathscr C$ with girth at least $g_{r+1}$ satisfies $\arb_r(G)\le r+1$.  Hence, it holds that   $\A_{r}(\mathscr{C})=r+1$.

If $\mathscr{C}$ contains no graphs $G$ with $r<\girth(G)<\infty$, then every graph $G\in \mathscr{C}$ with girth at least $r+1$ is a forest and thus $\arb_r(G)=1$. Hence, $\A_{r}(\mathscr{C})=1$.
\end{proof} 

Note that if a  proper minor-closed class of graphs   $\mathscr{C}$ does not contain some cycle  (say $C_{q}$),
then   every graph  $G$ in $\mathscr{C}$ with girth at least $q$  is a forest  and thus $\arb_{p}(G)=1$. 
Hence, one can deduce from \Cref{th:arb-subexponential} that

\begin{corollary}
Let $\mathscr{C}$    be a 	 proper minor-closed class of graphs and let $r$ be a positive integer.

If $\mathscr{C}$ contains all cycles, then  $\A_{r}(\mathscr{C})=r+1$; otherwise,  $\A_{r}(\mathscr{C})=1$.
\end{corollary}

 \subsection{Generalized $r$-acyclic chromatic index}
 \label{sec:acyclic}
 
 In this subsection, we  consider the generalized $p$-acyclic chromatic indices of large girth graphs. 
 In the following, we assume that  graphs have maximum degree $\Delta\ge3$ as the case  $\Delta\le2$ is trivial.
We begin with the following theorem.

\begin{theorem}\label{th:acyclic-pd}
    Let $r\ge3$ be an integer and let $G$ be a  $(r+1)$-path degenerate graph with maximum degree $\Delta\ge3$. 
    
    If $G$ is  a forest, then $\ap_{r}(G)=\Delta$; otherwise, $\ap_{r}(G)= \max\{\Delta,r\}$.
\end{theorem}
\begin{proof}
   If $G$ is a forest, obviously $\ap_{r}(G)=\Delta$.
   Suppose that $G$ contains at least a cycle.
Since $G$ is $(r+1)$-path degenerate, 
$\girth(G)>r+1$. Hence, by the definition of $\ap_{r}(G)$, we have
 $\ap_{p}(G)\ge \max\{\Delta,r\}$.

 We prove 
 by induction on the number of edges of $G$ that  $\ap_{r}(G)\le \max\{\Delta,r\}$.
  Note that the statement obviously holds if $G$ is a forest, and that we may assume that $G$ is connected. 
Since  $G$ is  $(r+1)$-path degenerate, there exists a $r$-reduction $R$ of $G$ which  deletes either  a vertex $v$ with degree 1 or  the internal vertices of a strict ear $P$ with length at least $r+1$.
Let $H$ be the resulting graph after applying  $R$ on $G$.
Clearly, we have $\max\{\Delta,r\}\le\max\{\Delta(H)+1,r\}$. 
By induction, there exists a $r$-acyclic edge  coloring $\phi$ of $H$ using at most $\max\{\Delta(H),r\}$ colors. 
 We will extend this coloring to a new coloring $\psi$ of $G$ using at most $\max\{\Delta,r\}$ colors. %

If $R$ deletes   a vertex $v$ with degree 1,
  assume that $G=H \cup \{uv\}$. 
  We extend $\phi$ to $\psi$ by coloring $uv$ properly (only if $d_{H}(u)=\Delta(H)\ge r$, color $uv$  with a now color).
Since every cycle $C$ of $G$ is contained in $H$, $\psi$ is obviously a $r$-acyclic edge  coloring of $G$ using $\max\{\Delta,r\}$ colors.

If $R$ deletes  the internal vertices of a strict ear $P$ with length at least $r+1$, let $P=u_1v_1v_2\ldots v_ju_2$ ($j\ge r$). 
We first color the two edges $u_{1}v_{1}$ and $u_{2}v_{j}$ properly based on $\phi$ (if necessary, use a new color).
Assume that 
$u_{1}v_{1}$ and $u_{2}v_{j}$ are 
colored with $\alpha$ and  $\beta$, respectively.
If $\alpha\neq\beta$,  we further  color the $r-2$ edges $v_1v_2,\ldots,v_{r-2}v_{r-1}$ edges with  the  $r-2$ colors  that are different from $\alpha$ and $\beta$ and then color the remaining $j+1-r$ ($>0$) edges of $P$ greedily; 
 otherwise,    we  color the remaining $j-1$ ($\ge r$) edges of $P$ with the  $r$ colors  that are different from $\alpha$.
 In both cases,  we obtain an edge coloring $\psi$ of $G$ using   $\max\{\Delta,p\}$ colors, in which  the edges of $P$ are colored with at least $p$ colors.
 Note that every cycle $C$ of $G$ is either contained in $H$ or  includes all the edges of $P$,  it gets at least  $\min\{|C|,r\}$ colors under $\psi$.  Therefore, $\psi$ is a  $r$-acyclic edge  coloring  using   $\max\{\Delta,r\}$ colors.
\end{proof}

Now,  \Cref{th:acyclic-pd}, together with \Cref{th:pd-subexponential}, implies the main thereom of this section, which we restate here.
\ThmAcyclic*

Similarly,
combining  \Cref{th:pd-polynomial} (resp. \Cref{th:pd-MC}, \Cref{th:pd-planar}) and  \Cref{th:acyclic-pd}, we immediately have the following three consequences.

\begin{corollary}
 \label{cor:acyclic-polynomial}
      Given an integer $r\ge1$.  Let $\mathscr C$ be a class with polynomial expansion, with ${\rm Exp}_{\mathscr C}(k)\leq a(k+\frac12)^b$ for some positive reals
 	$a,b$.
Let $G$ be a graph in $\mathscr C$ with
$$\girth(G)>\max\left(7,2\left\lfloor-\frac{2b}{\log 2}\W\left(-\frac{\log 2}{(24\sqrt{2}a)^{1/b}\,br}\right)\right\rfloor+4\right)p\sim 4br\log_2r.$$

   If $G$ is  a forest, then $\ap_{r}(G)=\Delta$; otherwise, $\ap_{r}(G)= \max\{\Delta,r\}$.
\end{corollary}
\begin{corollary}\label{cor:acyclic-PMCC}
    Given an integer $r\ge1$.	Let $\mathscr{C}$ be a proper minor-closed class, let $d=\mad(\mathscr{C})$, and let
	 $G$ be a graph in $\mathscr{C}$ with girth
     greater than
     $$\girth(G)>(4\log_{2}d+2\log_{2}(\min\{d,576\})+3)r.$$

   If $G$ is  a forest, then $\ap_{r}(G)=\Delta$; otherwise, $\ap_{r}(G)= \max\{\Delta,r\}$.
\end{corollary}

\begin{corollary}
\label{th:acyclic-planar}
Let $r\ge3$ be an integer. Let $G$  be a planar graph with  $\Delta\ge 3$ and $\girth(G)\ge5r+1$.

     If $G$ is  a forest, then $\ap_{r}(G)=\Delta$; otherwise, $\ap_{r}(G)= \max\{\Delta,r\}$. 
\end{corollary}

Remark that the case $p=3$ and the case $p\ge4$ of \Cref{th:acyclic-planar} were independently proved in \cite[Theorem 3.3]{HWLL2009} and
 \cite[Theorem 6]{ZWYLL2012}, respectively.

\subsection{Weak $r$-coloring number}
 \label{sec:weak}

 We begin with some definitions and some useful tools.

 Let $G$ be a graph and let $\Pi(G)$ be the set of all linear orders of $V(G)$.
Given  $\pi\in \Pi(G)$ and  a nonnegative integer $r$.
In the following, it will be convenient to identify a linear order of $V(G)$ with a word, whose letters are $V(G)$, using every $v\in V(G)$ exactly once. 
To facilitate reading,
we shall use the comma to denote 
concatenation. For instance, if $\pi$ is the linear order $b<c<d$ then $a,\pi,e$ denotes the linear order $a<b<c<d<e$.

 For any two vertices $u,v\in V(G)$,
 we say that $u$ is \emph{weakly $r$-reachable} from $v$ under $\pi$  if $u\leq_\pi v$ and there exists a $u$-$v$-path $P$   of length at most $r$  with $u<_\pi w$ for all internal vertices $w\in V(P)$. 
 Note that, every vertex of $G$ is $r$-weakly reachable from itself under $\pi$.
 We denote by $\WReach_r[G,\pi,v]$ the set of vertices that are weakly $r$-reachable from $v$ under $\pi$. 
 The \emph{weak $r$-coloring number} of
 $G$, denoted by $\wcol_r(G)$, is defined as
$$\wcol_r(G)=\min_{\pi\in\Pi(G)}\max_{v\in V(G)}|\WReach_r[G,\pi,v]|.$$
It is clear that $\wcol_0(G)=1$ and for any integer $x\ge 1$, $\wcol_x(G)\ge1$.

Let $f:\mathbb N\rightarrow\mathbb R$ be a function, let $r$ be an integer, and let $G$ be a graph.
A linear order $\pi_G$ of $V(G)$ is a
 \emph{good linear order with respect to $f$ up to $r$} if, for  every nonnegative integer $x\le r$ we have
 \[
 \max_{v\in V(G)} |\WReach_x[G,\pi_G,v]|\le f(x).
 \]
Note that if $f$ is a good linear order with respect to $f$ up to $r$, then 
$\wcol_x(G)\leq f(x)$ for every nonnegative integer $x\leq r$.

\begin{lemma}\label{lemma:wcol-01-vertex}
   Let $f:\mathbb{N}\to \mathbb{R}$ be an increasing function with $f(x+1)\geq f(x)+1$, let $r\geq 1$ be an integer, and let $G$ be a graph.  
Let $H$ be the graph obtained from $G$ by deleting a vertex  $v$ with degree  at most 1. If 
   $H$ is not empty and has a  good linear order $\pi_H$ 
   with respect to $f$ up to $r$, then   $G$  also has a  good linear order $\pi_G$ 
   with respect to $f$ up to $r$.
\end{lemma}
\begin{proof}
Let $\pi_G=\pi_H,v$. 
If $v$ is an isolated vertex of $G$, then $\pi_G$ is obviously  a good linear order of $V(G)$ with respect to $f$ up to $r$.
     Assume that $v$ is a vertex with degree 1 of $G$ and  $uv\in E(G)$. 
     For every integer $x\ge0$, it is clear that $|\WReach_x[G,\pi_G,w]|=|\WReach_x[H,\pi_H,w]|\le f(x)$ for every  $w\in V(H)$ and   $|\WReach_x[G,\pi_G,v]|=|\WReach_{x-1}[H,\pi_H,u]|+1\le f(x-1)+1\le f(x)$.
Since $H$ is not empty  and $\pi_H$ is a good lenear order of $H$ with respect to $f$ up to $r$, $f(0)\ge1$ and $f(1)\ge2$.
Thus, we have $|\WReach_0[G,\pi_G,v]|=1\le f(0)$, 
$|\WReach_1[G,\pi_G,v]|=2\le f(1)$ and $|\WReach_x[G,\pi_G,v]|\le f(x-1)+1\le f(x)$ for every integer $x\ge2$. Therefore,  $\pi_G$  is a good linear order of $G$ with respect to $f$  up to $r$.
\end{proof}

\begin{figure}
    \centering
    \includegraphics[width=\textwidth]{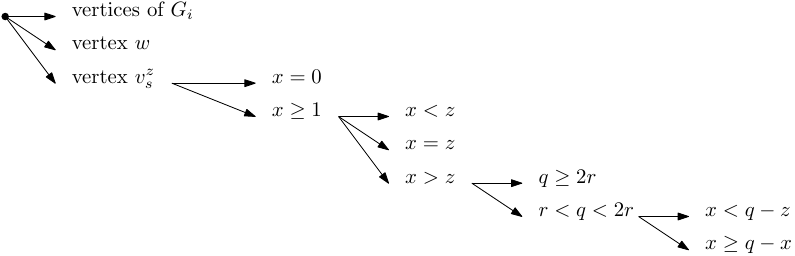}
    \caption{The different cases considered in the proof of \Cref{th:wocl-pd} to bound 
$|\WReach_x[G_{i+1},\pi_{i+1},v]|$.
    }
    \label{fig:cases}
\end{figure}

\begin{lemma}\label{th:wocl-pd}
	Let $r,q$ be positive integers with $q\ge r+1$, 
	 and let  $G$ be a  $2q$-path degenerate graph.
	Define $f(0)=1$ and, if $x\geq 1$, 
	\[
	f(x)=\begin{cases}
		x+2+\log_2\bigl(\frac{q-1}{q-x}\big),&\text{if }q <2r,\\
		x+2&\text{otherwise.}
	\end{cases}
	\]
	Then, $G$ has a good linear order $\pi$ with respect to $f$ up to $r$.
\end{lemma}
\begin{proof}
	Let $p=2q$.
	According to the definition of path degeneracy and  \Cref{rem:exact}, there exists a sequence $G_1 \subset \ldots \subset G_t=G$ such that $G_1$ is an isolated vertex, and  each $G_i$ ($i\ge1$) is obtained from $G_{i+1}$ by performing a $p$-reduction $R_{i+1}$ of $G_{i+1}$, where $R_{i+1}$ deletes either a vertex with degree at most $1$ or a strict ear with length  $p$
	of $G_{i+1}$
	
	We prove that $G_i$ has a good linear order with respect to $f$ up to $r$ by induction on $i$ (for $1\leq i\leq t$).
	The base case $i=1$ is trivial since $G_1$ is an isolated vertex. 
	Assume that every graph $G_j$ ($1\le j \le i$) has  a good linear order $\pi_j$ with respect to $f$  up to $r$. We prove that $G_{i+1}$  has  a good linear order $\pi_{i+1}$ with respect to $f$  up to $r$.   
	We may assume that $R_{i+1}$ deletes  a strict ear $P$ with length  $p$ of $G_{i+1}$ as otherwise the statement is true according to \Cref{lemma:wcol-01-vertex}.

	Let $u_1$ and $u_2$ be the two endpoints of $P$.
	We denote by 
	\[u_1,v^1_1,\dots,v^1_{q-1},w,v^2_{q-1},\dots,v^2_1,u_2\]
	 (in order) the vertices of $P$, and consider the linear order
	\[
	\pi_{i+1}:= w,\pi_i,v^1_1,\dots,v^1_{q-1},v^2_1,\dots,v^2_{q-1}
	\]
	of $V(G_{i+1})$.

    We aim to prove that $|\WReach_x[G_{i+1},\pi_{i+1},v]|\le f(x)$ holds for all vertices $v$ of $G_{i+1}$. We will consider several cases, as indicated on \Cref{fig:cases}.
	\medskip
    
	Since every vertex $v\in V(G_i)$ is at distance at least $r+1$  from $w$, we have 
	$$|\WReach_x[G_{i+1},\pi_{i+1},v]|=|\WReach_x[G_{i},\pi_{i},v]|\le f(x).$$
	
	Moreover, for every $0\le x \le r$, it is clear that $$|\WReach_x[G_{i+1},\pi_{i+1},w]|=1\le f(x).$$
	
	It remains to consider vertices $v^s_{z}$ (for $s\in\{1,2\}$ and $1\le z \le q-1$).
	
	If $x=0$, then $|\WReach_0[G_{i+1},\pi_{i+1},v^s_z]|=1= f(0)$. Hence, we can assume $x\geq 1$.
	We consider three different cases.
	\medskip
	
	$\bullet$ If $x<z$, then 	\[
	\WReach_x[G_{i+1},\pi_{i+1},v^s_z]=
	\begin{cases}
		\{v^s_{z-x},\dots,v^s_z\},&\text{if }x< q-z,\\
		\{v^s_{z-x},\dots,v^s_z,w\},&\text{otherwise}.
	\end{cases}
	\]
	Hence,
	\[	\max_{1\le x <z<q}|\WReach_x[G_{i+1},\pi_{i+1},v^s_z]|=x+2\le f(x). 
    \]
	\medskip

$\bullet$ If $x=z$, we have
$|\WReach_x[G_{i+1},\pi_{i+1},v^s_x]|\leq x+2\leq f(x)$.
\medskip

$\bullet$ Thus, we can assume $x>z$.
	
	It is easily checked that $\WReach_x[G_{i+1},\pi_{i+1},v^s_z]$ is the disjoint union of 
	$\WReach_{x-z}[G_i,\pi_i,u_s]$  and a subset $N_{x,s,z}$ of internal vertices of $P$ defined by
	\[
	N_{x,s,z}=
	\begin{cases}
		\{v^s_1,\dots,v^s_z\},&\text{if }x< q-z,\\
		\{v^s_1,\dots,v^s_z,w\},&\text{otherwise}.
	\end{cases}
	\] 
	
	As $\pi_i$ is a good linear order with respect to $f$ up to $r$ for $G_i$, we deduce that for every $s\in\{1,2\}$ and $1\le z < q$ we have
	\[
	|\WReach_x[G_{i+1},\pi_{i+1},v^s_z]|\leq f(x-z)+z+\begin{cases}
		0,&\text{if }x< q-z,\\
		1,&\text{if }x\geq q-z.
	\end{cases}
	\]
	
\begin{itemize}
\renewcommand{\labelitemi}{--}
\item Assume $q\ge 2r$ (in which case $f(x)=x+2$). 

Then, there is no $x\leq r$ and $z<x$ such that 
$x\geq q-z$. Thus, we have
\[
|\WReach_x[G_{i+1},\pi_{i+1},v^s_z]|\leq f(x-z)+z=f(x).\]
Hence, $\pi_{i+1}$ is a good linear order of $G_{i+1}$ with respect to $f$ up to $r$.

\item Otherwise,  $r+1\leq q<2r$  (in which case $f(x)=x+2+\log_2\bigl(\frac{q-1}{q-x}\big)$).

If $x<q-z$, we have
\[|\WReach_x[G_{i+1},\pi_{i+1},v^s_z]|\leq f(x-z)+z\leq f(x).\]
Otherwise, we have
$z\ge q-x$ (hence $q-x+z\ge 2(q-x)$), and thus
\begin{align*}
 \qquad\quad\ \mathrlap{|\WReach_x[G_{i+1},\pi_{i+1},v^s_z]|}\phantom{|\WReach_x[G_{i+1},}&\phantom{\pi_{i+1},v^s_z]|}\leq
f(x-z)+z+1\\
&=x-z+2+\log_2(q-1)-\log_2(q-x+z)+z+1\\
&\le x+2+\log_2(q-1)-\log_2(2(q-x))+1\\
&=x+2+\log_2(q-1)-\log_2(q-x)\\
&=f(x)
\end{align*}
Hence, $\pi_{i+1}$ is a good linear order of $G_{i+1}$ with respect to $f$ up to $r$.
\end{itemize}
\end{proof}

We derive bounds for the weak coloring numbers for graphs of high girth in classes with sub-exponential expansion.

\ThmWcol*
\begin{proof}
This is a direct consequence of 
\Cref{th:pd-subexponential,th:wocl-pd}.  
\end{proof}

In particular, \Cref{th:wcol-subexponential} gives the following bounds for the two extreme cases of $q=r+1$ and $q=2r-1$:
\begin{align*}
    \wcol_r(G)\leq\begin{cases}
        r+2+\lfloor
        \log_2 r\rfloor&\text{if }
        \girth(G)\ge g(2r+2),\\
    r+3&\text{if }\girth(G)\ge g(4r-2).
    \end{cases}
\end{align*}

Recall that even for the class of planar graphs $\mathscr{C}$, 
the  lower and upper bounds on $\max_{G\in\mathscr{C}} wcol_r(G)$ are $\Omega(r^2\log r)$ \cite{JM2021} and $O(r^3)$ \cite{VOQRS2017}, respectively.
\Cref{th:wcol-subexponential} indicates that 
for any 
 class $\mathscr{C}$ with sub-exponential expansion, the weak $r$-coloring number of large girth 
graphs in $\mathscr{C}$ is  close to $r+2$, which is the best possible upper bound on $\wcol_r(G)$.

\bibliographystyle{amsplain}
\bibliography{ref}

\end{document}